\documentclass[11pt,a4paper]{amsart}

\usepackage{amsmath}
\usepackage{amsfonts}
\usepackage{amssymb}
\usepackage{amsthm}
\usepackage{epsfig,subfig}
\usepackage{graphicx}
\usepackage[usenames,dvipsnames]{xcolor}
\usepackage[colorlinks=true,linkcolor=Blue,citecolor=Green]{hyperref}
\usepackage{url}
\usepackage{color}

\newtheorem{theorem}{Theorem}[section]

\newtheorem{remark}{Remark}[section]

\DeclareMathOperator{\inter}{int}
\def\lin{\mathop\mathrm{lin}\nolimits}

\def\conv{\mathop\mathrm{conv}\nolimits}

\def\K{\mathcal{K}}

\def\R{\mathbb{R}}
\def\N{\mathbb{N}}
\def\inter{\mathrm{int}}
\def\e{\mathrm{e}}

\def\Z{\mathbb{Z}}
\def\A{\mathrm{A}}

\def\p{\mathrm{p}}
\def\D{\mathrm{D}}

\def\cir{\mathrm{R}}

\def\inr{\mathrm{r}}

\def\e{\mathrm{e}}

\numberwithin{equation}{section}

\begin{document}

\title[A remark on inequalities under lattice constraints]{A remark on perimeter-diameter and
perimeter-circumradius inequalities under lattice constraints}

\author{Bernardo Gonz\'alez Merino \and Matthias Henze}

\address{Departamento de Matem\'aticas, Universidad de Murcia, Campus Espinar\-do, 30100-Murcia, Spain}
\email{bgmerino@um.es}
\address{Institut f\"ur Informatik, Freie Universit\"at Berlin, Takustra\ss e 9, 14195 Berlin, Germany}
\email{matthias.henze@fu-berlin.de}

\thanks{BM was supported by MINECO project
MTM2012-34037 and by ``Programa de Ayudas a Grupos de Excelencia de la
Regi\'on de Murcia'', Fundaci\'on S\'eneca, 04540/GERM/06 and MH by
the ESF EUROCORES programme EuroGIGA-VORONOI, (DFG): Ro 2338/5-1.}



\begin{abstract}
In this note, we study several inequalities involving geometric functionals for lattice point-free planar convex sets.
We focus on the previously not addressed cases perimeter--diameter and perimeter--circumradius.
\end{abstract}

\maketitle

\section{Introduction}
Let $\K^2$ be the set of all planar closed convex sets and denote by $\Z^2$ the standard integer lattice in $\R^2$.
Some $K\in\K^2$ is called \emph{lattice-free} if $\inter K\cap\Z^2=\emptyset$, that is,
the interior of $K$ does not contain any lattice point of~$\Z^2$.

The perimeter, diameter, circumradius, inradius, minimal width and the
area of a convex body $K\in\K^2$ are denoted by $\p(K)$, $\D(K)$,
$\cir(K)$, $\inr(K)$, $\omega(K)$ and $\A(K)$, respectively.
The study of optimal relations between two of these functionals (for convex sets of arbitrary dimension) is a classical problem in Convex Geometry (cf.~\cite[pp.~56--59]{BF1987}).

In the planar case, there is an extensive bibliography if one adds the extra assumption
that $K$ is lattice-free (cf.~\cite{CFG1994,EGH1989,GW1993,H1977,HCS1998,S1988}).
For this situation, Hillock \& Scott~\cite{HS2002} collected the known best possible inequalities relating pairs of the six functionals above.

The only pairs that are missing in their list are $(\p,\D)$ and $(\p,\cir)$.
They have not been addressed so far and are the subject of our interest.
The fact that lattice-freeness is not preserved by arbitrary scaling is usually reflected in the non-homogeneity of the geometric inequalities that are derived.
In this spirit, we propose the study of sharp upper bounds for the non-negative
functionals $\p(K)-2\D(K)$ and $\p(K)-4\cir(K)$, for lattice-free~$K\in\K^2$.
The existence of such upper bounds is proven by
\[
\p(K)-4\cir(K)\leq\p(K)-2\D(K)\leq 2.65,
\]
which follows from $\sqrt{3}\D(K)(\p(K)-2\D(K))\leq 4\A(K)$ (see~\cite{K1924}) together with $\A(K)\leq\lambda\D(K)$ (see~\cite{S1974}), $\lambda\approx 1.144$.

We conjecture, however, that the following bounds are the best possible
\begin{align}
&\p(K)-2\D(K)\leq1+\frac{2}{\sqrt{3}}\approx2.1547\quad\text{and}\quad\p(K)-4\cir(K)\leq 2.\label{eq:optimal}
\end{align}
The equilateral triangle of edge lengths $1+2/\sqrt{3}$ for the pair $(\p,\D)$
and the split $\{x\in\R^2:0\leq x_2\leq1\}$ for the pair $(\p,\cir)$ attain equality.

In the following, we prove our conjectured inequalities in various cases, and offer sharp bounds on some non-linear functionals related to these magnitudes.
A general proof for~\eqref{eq:optimal} has to be left as an open problem.

For our first result, we need to recall the notion of an \emph{unconditional} set:
some $K\in\K^2$ that is symmetric with respect to the lines $z+\lin\{\e_1\}$ and $z+\lin\{\e_2\}$, for a suitable $z\in\R^2$.

\begin{theorem}\label{th:pRC}
Let $K\in\K^2$ be lattice-free and unconditional.
Then
\begin{equation}\label{eq:pRC}
\p(K)-2\D(K)=\p(K)-4\cir(K)\leq2.
\end{equation}
The inequality is best possible.
\end{theorem}

Often one can apply appropriate Steiner symmetrizations to a general lattice-free~$K$
to obtain a lattice-free unconditional set (cf.~\cite{S1974}).
Unfortunately, this method usually decreases the functional $\p(K)-4\cir(K)$ and hence
is not applicable in our situation.

Our second result shows the validity of the first conjectured inequality in~\eqref{eq:optimal} for triangles.

\begin{theorem}\label{th:pDC}
Let $T\in\K^2$ be a lattice-free triangle.
Then
\begin{equation}\label{eq:pDC}
\p(T)-2\D(T)\leq\frac2{\sqrt{3}}\left(1+\frac{\omega(T)}{\D(T)}\right).
\end{equation}
In particular, $\p(T)-2\D(T)\leq1+2/\sqrt{3}$, and
equality holds in~\eqref{eq:pDC} if and only if $T$ is an equilateral triangle with edge lengths $1+2/\sqrt{3}$.
\end{theorem}

Note that the refined inequality~\eqref{eq:pDC} is specific to triangles and does not hold for general lattice-free convex sets.

Complementing the partial results above, we found the following sharp, yet weaker inequalities relating
the magnitudes of interest.

\begin{theorem}\label{th:pRtight_pDtight}
Let $K\in\K^2$ be lattice-free. Then
\begin{itemize}
 \item[i)] $\frac{\D(K)-1}{\D(K)}(\p(K)-2\D(K))<2$,
 \item[ii)] $\frac{2\cir(K)-1}{2\cir(K)}(\p(K)-4\cir(K))<2$.
\end{itemize}
None of the inequalities can be improved.
\end{theorem}

Observe that our conjectured bound for the pair $(\p,\D)$ in~\eqref{eq:optimal} is independent from inequality i) above,
whereas the conjectured bound for $(\p,\cir)$ would strengthen inequality ii) by
$\frac{2\cir(K)-1}{2\cir(K)}(\p(K)-4\cir(K))\leq\frac{2\cir(K)-1}{2\cir(K)}\cdot2<2$.

\section{Proofs of the inequalities}

\begin{proof}[Proof of Theorem \ref{th:pRC}]
First of all, since $K$ is unconditional we have $\D(K)=2\cir(K)$ and it suffices to show the inequality $\p(K)-4\cir(K)\leq2$.

Let $z\in\R^2$ be the center of symmetry of $K$.
Note, that $z$ lies in the interior of $K$ and is at the same time its circumcenter.
As~$\Z^2$ is symmetric with respect to the coordinate axes, we may assume that after suitable reflections and translations of $K$ its center $z$ is contained in $[0,1/2]^2$.

Since $0\notin\inter K$, there exists a supporting line $L$ of $K$ with $0\in L$.
We first suppose that $L\cap[0,1]^2\neq\{0\}$.
Since $z\in[0,1/2]^2$, it holds $d(z,L)=\min_{y\in L}\|z-y\|\leq 1/2$, where $\|\cdot\|$ denotes the Euclidean norm.
Due to the unconditionality of~$K$, the symmetric line $L'$ to $L$ with respect to $z$ supports~$K$ as well.
Therefore, $K$ is contained in the strip determined by $L$ and $L'$ which has width at most~$1$, hence $\inr(K)\leq1/2$.
Using an inequality of Henk \& Tsintsifas~\cite{HT1994}, we get $\p(K)\leq4\cir(K)+4\inr(K)\leq4\cir(K)+2$, as desired.

We now consider the case $L\cap[0,1]^2=\{0\}$.
We shoot a ray from $z$ in direction $(-1,-1)$ and we let $q\in L$ be the intersection point of this ray and $L$.
Since $L$ has negative slope, $q_1\geq0$ if $z_1\geq z_2$, and $q_2\geq0$ if $z_1\leq z_2$.
In both cases, it follows that $0\neq\lambda=\|z-q\|\leq1/\sqrt{2}$.
Let $K':=(\lambda \sqrt{2})^{-1}(-q+K)$.
The functionals $\p$ and $\cir$ are homogeneous of degree $1$, and so
\[
\p(K)-4\cir(K)\leq(\lambda \sqrt{2})^{-1}(\p(K)-4\cir(K))=\p(K')-4\cir(K').
\]
We observe that $K'$ is unconditional with respect to $(\lambda\sqrt{2})^{-1}(-q+z)=(1/2,1/2)$,
and the line $(\lambda\sqrt{2})^{-1}(-q+L)=L$ supports $K'$.
Moreover, the unconditionality of $K'$ implies that the lines $L_1,L_2$, and $L_3$ symmetric to $L$,
with respect to $(1/2,1/2)$, $(1/2,1/2)+\lin\{\e_1\}$, and $(1/2,1/2)+\lin\{\e_2\}$, respectively, support $K'$.
Thus $K'\subseteq Q$, where $Q$ is the rhombus determined by these four lines and therefore $K'$ is lattice-free.
By definition of the circumradius, we have $K'\subseteq (1/2,1/2)+\cir(K')[-1,1]^2=Q'$.
Thus $K'\subseteq Q\cap Q'$ and hence $\p(K')\leq\p(Q\cap Q')$.

In the last step, we show
\begin{align}
\p(Q\cap Q')&\leq\p\left((1/2,1/2)+[-\cir(K'),\cir(K')]\times[-1/2,1/2]\right)\label{eq:uncond_1}\\
&=4\cir(K')+2,\nonumber
\end{align}
which implies the desired inequality (see Figure~\ref{fig:uncond}).
To this end, we remark that the four vertices of $Q$ cannot all lie in $\inter Q'$,
as this would mean that $\cir(Q)<\cir(K')\leq\cir(Q)$, a contradiction.
Thus, we assume without loss of generality that the two vertices of $Q$ that are contained in the line $(1/2,1/2)+\lin\{\e_2\}$ lie outside of $\inter Q'$.

Let $N$ be the intersection point of $L$ and the boundary of $Q'$ with $N_1\leq0$, and let
$M$ be the intersection point of $L$ with the boundary of $[0,1/2]\times[-\cir(K')+1/2,0]$ with $M_1\geq 0$.
Moreover, we define the following distances of segments in $Q$ and $Q'$ (see Figure~\ref{fig:uncond}):
\[
a=\|M\|, A=|M_1|, B'=|M_2|,\quad\textrm{and}\quad b=\|N\|, B=|N_1|, C=|N_2|.
\]

\begin{figure}[ht]
\centering
\includegraphics[width=8cm]{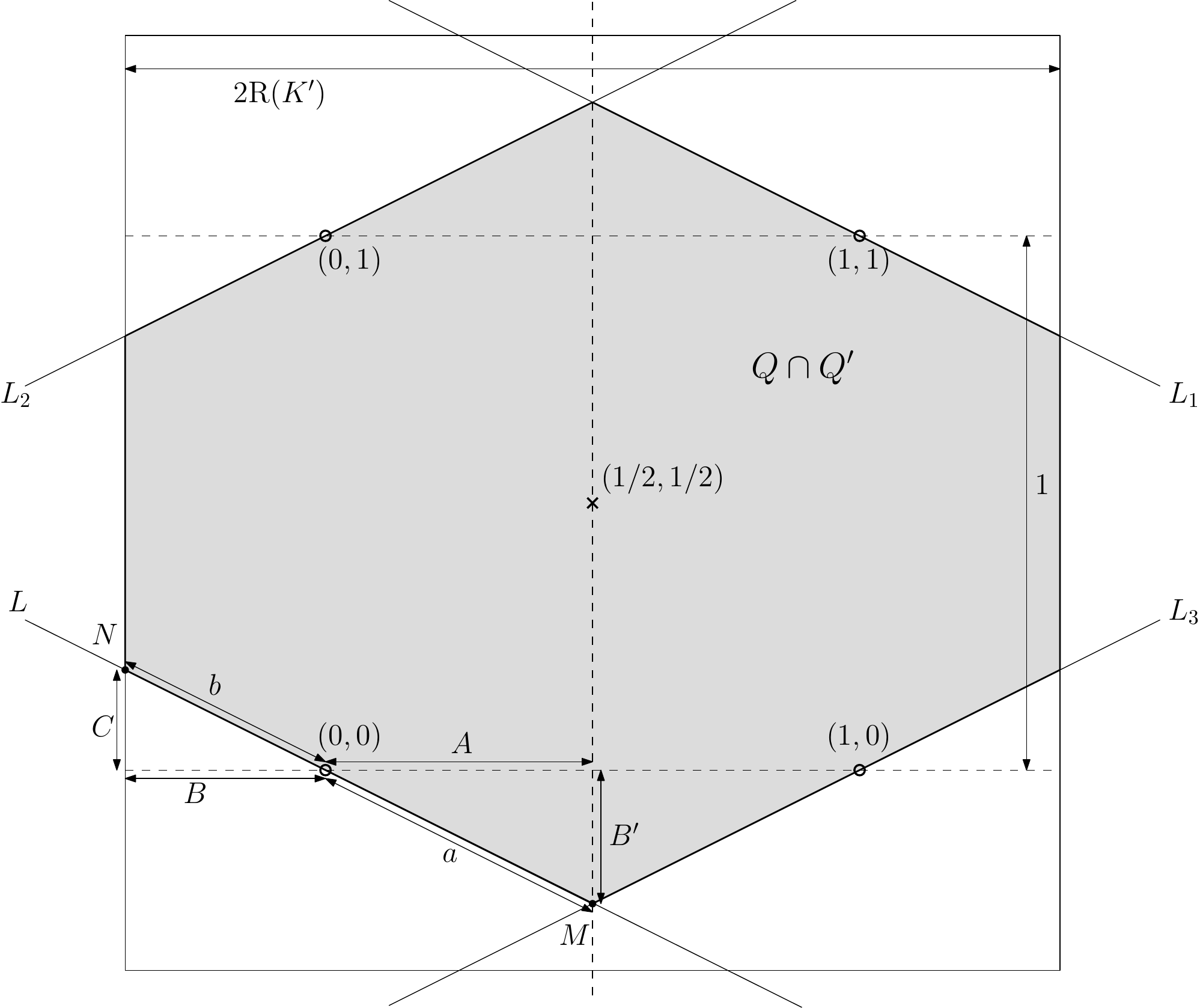}
\caption{Bounding the perimeter of $Q\cap Q'$.}
\label{fig:uncond}
\end{figure}

By the symmetry of $Q\cap Q'$, it is enough to prove $a+b\leq A+B+C$ in order to get~\eqref{eq:uncond_1}.
Using basic properties of homothetic triangles and Pythagoras' theorem, we obtain
\[
B'\leq B,\quad \frac{a}{A}=\frac{b}{B},\quad
b^2=B^2+C^2,\quad\frac{C}{B}=\frac{B'}{A}.
\]
Writing $a=bA/B$ and $b=\sqrt{B^2+C^2}$, the
inequality $a+b\leq A+B+C$ becomes
\[
\sqrt{B^2+C^2}(A+B)\leq B(A+B+C).
\]
Since $C=BB'/A$, this is equivalent to
\[
\sqrt{A^2+(B')^2}(A+B)\leq A^2+AB+BB'.
\]
Taking squares on both sides gives
\[
AB'+2BB'\leq 2AB+2B^2,
\]
which follows from $B'\leq B$.
Therefore, inequality~\eqref{eq:uncond_1} holds and we have $\p(K)-4\cir(K)\leq\p(Q\cap Q')-4\cir(K')\leq 2$.
\end{proof}

\begin{remark}
The first part of the above proof shows that, in general, if $\inr(K)\leq 1/2$ for some $K\in\K^2$, then $\p(K)-4\cir(K)\leq2$.
\end{remark}

\begin{proof}[Proof of Theorem~\ref{th:pDC}]
We start by determining the scaling factor $\lambda>0$ for which $T'=\lambda T$ is such that the
length of the segment $T'\cap L$ is equal to~$1$, where $L$ is the line that is parallel and at distance $1$
to the longest edge $e$ of~$T'$ and on the same side of $e$ as the vertex of $T'$ that is not contained in~$e$ (see Figure~\ref{fig:triangle}).

\begin{figure}[ht]
\centering
\includegraphics[width=6cm]{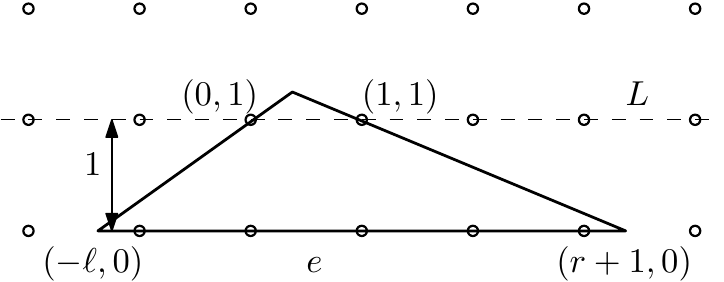}
\caption{The triangle $T'$.}
\label{fig:triangle}
\end{figure}

Since the diameter of $T'$ is attained by its longest edge, we get from Thales' Theorem that
\[
\frac{1}{\lambda\omega(T)-1}=\frac{1}{\omega(T')-1}=\frac{\D(T')}{\omega(T')}=\frac{\lambda\D(T)}{\lambda\omega(T)}=\frac{\D(T)}{\omega(T)},
\]
and thus $\lambda=(\omega(T)+\D(T))/(\omega(T)\D(T))$.
Scott~\cite{S1978} showed that for lattice-free $T$ it holds $(\omega(T)-1)(\D(T)-1)\leq 1$.
This is equivalent to $\omega(T)\D(T)\leq\omega(T)+\D(T)$ and hence $\lambda\geq 1$.
Therefore, we have $\p(T)-2\D(T)\leq\lambda(\p(T)-2\D(T))=\p(T')-2\D(T')$ and we can restrict
our attention to the triangle $T'$.

Now, we rotate and translate $T'$ appropriately such that its longest
edge lies on the $x$-axis and the chord $T'\cap L$ has endpoints $(0,1)$ and $(1,1)$.
Let us further denote the vertices of the longest edge by $(-\ell,0)$ and $(r+1,0)$, for $\ell,r\geq0$,
and we may assume that $\ell\leq r$ (see~Figure~\ref{fig:triangle}).
A straightforward computation shows that
$\left(\frac{\ell}{\ell+r},\frac{\ell+r+1}{\ell+r}\right)$ is the third vertex of
$T'$, and moreover $\omega(T')=\frac{\ell+r+1}{\ell+r}$ and $\D(T')=\ell+r+1$.
The vertices $\left(\frac{\ell}{\ell+r},\frac{\ell+r+1}{\ell+r}\right)$ and $(r+1,0)$ determine an
edge of length at most $\D(T')$, and thus
\[
\ell+r+1\geq\sqrt{\left(r+1-\frac{\ell}{\ell+r}\right)^2+\left(\frac{\ell+r+1}{\ell+r}\right)^2}.
\]
Taking squares and dividing by $(\ell+r+1)^2$ we obtain $(\ell+r)^2\geq r^2+1$, and hence $\ell\geq\sqrt{r^2+1}-r$.
Together with $\ell\leq r$, this gives $r\geq 1/\sqrt{3}$.

As $p(T')-2\D(T')$ equals the sum of the short edges minus $\D(T')$, we get
\begin{align*}
\p(T')-2\D(T')&=\frac{\ell+r+1}{\ell+r}\left(\sqrt{r^2+1}-r+\sqrt{\ell^2+1}-\ell\right)\\
&=\omega(T')\left(\sqrt{r^2+1}-r+\sqrt{\ell^2+1}-\ell\right).
\end{align*}
Since $f(r)=\sqrt{r^2+1}-r$ is non-increasing and $\ell\geq\sqrt{r^2+1}-r$,
we get an upper bound on $p(T')-2\D(T')$ by substituting $\ell$ by $\sqrt{r^2+1}-r$ as follows
\[
\p(T')-2\D(T')\leq\omega(T')\sqrt{(\sqrt{r^2+1}-r)^2+1}.
\]
Now, we define $g(r)=\sqrt{f(r)^2+1}$ and we compute that
\[
g'(r)=\frac{\left(\frac{r}{\sqrt{r^2+1}}-1\right)\left(\sqrt{r^2+1}-r\right)}{\sqrt{(r-\sqrt{r^2+1})^2+1}}\leq0.
\]
Therefore, $g(r)$ is non-increasing as well, and by $r\geq1/\sqrt{3}$, we have
$g(r)\leq g(1/\sqrt{3})=2/\sqrt{3}$.
Using the formula for the scaling factor $\lambda$, we arrive at
\begin{align}
\p(T)-2\D(T)&\leq\p(T')-2\D(T')\leq\frac2{\sqrt{3}}\omega(T')\label{eq:proof_pDC}\\
&=\frac2{\sqrt{3}}\lambda\omega(T)=\frac2{\sqrt{3}}\left(1+\frac{\omega(T)}{\D(T)}\right).\nonumber
\end{align}
It is easy to see that $\omega(T)\leq\sqrt{3}/2\,\D(T)$ and hence $\p(T)-2\D(T)\leq1+2/\sqrt{3}$.
 
Tracing back the inequalities, we see that equality holds in~\eqref{eq:proof_pDC} if and only if
$\lambda=1$ and $\ell=r=1/\sqrt{3}$.
This means that $T$ is similar to the triangle with vertices $(-1/\sqrt{3},0),(1+1/\sqrt{3},0)$, and $(1/2,1+\sqrt{3}/2)$.
This triangle is equilateral with edge lengths $1+2/\sqrt{3}$.
\end{proof}

\begin{proof}[Proof of Theorem~\ref{th:pRtight_pDtight}]
The claimed inequalities are direct consequences of
$(2\inr(K)-1)(\D(K)-1)<1$ (see~\cite{AS1996}),
$(2\inr(K)-1)(2\cir(K)-1)<1$ (see~\cite{SA1999}), and
$\p(K)\leq 2\D(K)+4\inr(K)\leq4\cir(K)+4\inr(K)$ (see~\cite{HT1994}).

We may assume, that $\D(K)>1$ and $\cir(K)>\frac12$, respectively, since i) and ii) are otherwise certainly true.
Now, we have
\[
\p(K)\leq2\D(K)+4\inr(K)<2\D(K)+2\,\frac{\D(K)}{\D(K)-1},
\]
which shows i), and part ii) follows analogously from
\[
\p(K)\leq4\cir(K)+4\inr(K)<4\cir(K)+2\,\frac{2\cir(K)}{2\cir(K)-1}.
\]

Let's see why the inequalities are tight.
Let $K_n=\conv\{(\pm n,0),(\pm n,1)\}$, for $n\in\N$.
Clearly, $K_n$ is lattice-free, $\D(K_n)=2\cir(K_n)$, and for $n\to\infty$,
\begin{align*}
\frac{\D(K_n)-1}{\D(K_n)}&(\p(K_n)-2\D(K_n))=\frac{2\cir(K_n)-1}{2\cir(K_n)}(\p(K_n)-4\cir(K_n))\\
&=\frac{2\sqrt{n^2+\frac14}-1}{2\sqrt{n^2+\frac14}}\left(4n+2-4\sqrt{n^2+\frac14}\right)\nearrow2.\qedhere
\end{align*}
\end{proof}

\subsection*{Acknowledgements}
The second author gratefully acknowledges the hospitality of the group
Convex and Differential Geometry at the University of Murcia where part of this research was carried out.
We thank Mar\'ia Hern\'andez Cifre for valuable comments and fruitful discussions.

\providecommand{\bysame}{\leavevmode\hbox to3em{\hrulefill}\thinspace}
\providecommand{\MR}{\relax\ifhmode\unskip\space\fi MR }
\providecommand{\MRhref}[2]{%
  \href{http://www.ams.org/mathscinet-getitem?mr=#1}{#2}
}
\providecommand{\href}[2]{#2}

\end{document}